\documentclass[letterpaper,11pt]{article}
\usepackage{graphicx}
\usepackage{psfrag}
\usepackage{epsf}
\usepackage{amsmath,amsfonts,amssymb,latexsym}
\usepackage{enumitem}

\newtheorem{theorem}{Theorem}
\newtheorem{corollary}[theorem]{Corollary}
\newtheorem{lemma}[theorem]{Lemma}
\newtheorem{proposition}[theorem]{Proposition}

\newtheorem{question}[theorem]{Question}

\title{Infinite families of crossing-critical graphs\\ with prescribed average degree and crossing number}

     \author{\textsc{Drago Bokal}\footnotemark[2] \\[0.25em]
     {\small{Department of Mathematics}} \\[-0.25em]
     {\small{Institute of Mathematics, Physics, and Mechanics}} \\[-0.25em]
     {\small{Ljubljana, Slovenia}} \\[-0.1em]
     {\small\texttt{drago.bokal@imfm.uni-lj.si}}
     }


\footnotetext[2]{Supported in part by the Ministry of Education, Science and
Sport of Slovenia, Research Program P1--0507--0101.}

\date{\LaTeX-ed: \today}

\newcommand{\DEF}[1]{{\em #1\/}}
\newcommand{\dfnc}[3]{#1:#2\rightarrow #3}
\newcommand{\dset}[2]{\left\{#1 \:|\: #2\right\}}
\newcommand{\lset}[2]{\left\{#1, \ldots, #2\right\}}

\newcommand{\id}{\operatorname{id}}
\newcommand{\crn}{\operatorname{cr}}

\newcommand\tcrn{{\operatorname{tcr}}}

\newcommand{\eopf}{\raisebox{0.8ex}{\framebox{}}}

\newcommand{\NN}{\mathbb N}

\newcommand{\RR}{\mathbb R}

\newcommand{\cP}{{\cal P}}
\newcommand{\cQ}{{\cal Q}}

\newcommand{\cR}{{\cal R}}
\newcommand{\cT}{{\cal T}}
\newcommand{\cS}{{\cal S}}
\newcommand{\cH}{{\cal H}}
\newcommand{\cF}{{\cal F}}
\newcommand{\lw}{{\lambda}}
\newcommand{\rw}{{\rho}}

\newcommand{\GsG}{G_1\,\odot_{\sigma}G_2}
\newcommand{\GvvG}{G_1\,{}_{v_1}\!\!\odot_{v_2}G_2}

\newcommand{\vzip}[4]{#1\,{}_{#2}\!\!\odot_{#3}#4}
\newcommand{\crt}[2]{{#1}\;\square\;#2}

\newenvironment{proof}%
{\noindent{\bf Proof.}\ }%
{\hfill\eopf\par\bigskip}%

\begin{document}

\maketitle

\begin{abstract}
Šir\'{a}\v{n} constructed infinite families of $k$-crossing-critical graphs for every $k\ge 3$ and Kochol constructed such families of simple graphs for every $k\ge 2$. Richter and Thomassen argued that, for any given $k\ge 1$ and $r\ge 6$, there are only finitely many simple $k$-crossing-critical graphs with minimum degree $r$. Salazar observed that the same argument implies such a conclusion for simple $k$-crossing-critical graphs of prescribed average degree $r > 6$. He established existence of infinite families of simple $k$-crossing-critical graphs with any prescribed rational average degree $r\in [4,6)$ for infinitely many $k$ and asked about their existence for $r\in (3,4)$. The question was partially settled by Pinontoan and Richter, who answered it positively for $r\in(3\tfrac{1}{2},4)$. 

The present contribution uses two new constructions of crossing critical simple graphs along with the one developed by Pinontoan and Richter to unify these results and to answer Salazar's question by the following statement: for every rational number $r\in (3,6)$ there exists an integer $N_r$, such that, for any $k>N_r$, there exists an infinite family of simple $3$-connected crossing-critical graphs with average degree $r$ and crossing number $k$. Moreover, a universal lower bound on $k$ applies for rational numbers in any closed interval $I\subset (3,6)$. 

\end{abstract}


\section{Introduction}
\label{sc:introduction}

Let $\crn(G)$ denote the crossing number of a graph $G$. A graph $G$ is \DEF{$k$-crossing-critical}, if $\crn(G)\ge k$ and $\crn(G-e)<k$ for any edge $e\in E(G)$. Note that unless stated otherwise, all graphs in this paper are without vertices of degree two, as such vertices are trivial with respect to crossing number. The graphs may contain multiple edges, but do not contain loops. Besides that, the standard terminology from \cite{D} is used. 

\newpage 
Crossing-critical graphs give insight into structural properties of the crossing number invariant and have thus generated considerable interest. Šir\'{a}\v{n} introduced \DEF{crossing-critical edges} and proved that any such edge $e$ of a graph $G$ with $\crn(G-e)\le 1$ belongs to a Kuratowsky subdivision in $G$ \cite{S96}. Moreover, such a claim does not hold for edges with $\crn(G-e)\ge 5$. In \cite{S100}, Šir\'{a}\v{n} constructed the first infinite family of $3$-connected $k$-crossing-critical graphs for arbitrary given $k\ge 3$. Kochol constructed the first infinite family of simple $3$-connected $k$-crossing-critical graphs ($k\ge 2$) in \cite{K119}. Richter and Thomassen proved that $\crn(G)\le \frac{5}{2}k+16$ for a $k$-crossing-critical graph $G$ in \cite{RT198}. They used this result to prove that there are only finitely many simple $k$-crossing-critical graphs with minimum degree $r$ for any integers $k\ge 1$ and $r\ge 6$ and  constructed an infinite family of simple $4$-regular $4$-con\-nect\-ed $3$-crossing-critical graphs and posed a question about existence of simple $5$-regular $k$-crossing-critical graphs. Salazar observed that their argument implies finiteness of the number of the number of simple $k$-crossing critical graphs of average degree $r$ for any rational $r>6$ and integer $k>0$ \cite{S395}. Since the finiteness of the set of simple $3$-regular $k$-crossing-critical graphs can be established using Robertson-Seymour graph minor theory, it follows that the only average degrees for which an infinite family of simple $k$-crossing-critical graphs could exist are $r\in(3,6]$. Salazar constructed an infinite family of simple $k$-crossing-critical graphs with average degree $r$ for any $r\in [4,6)$ and posed the following question:
\begin{question}[\cite{S395}]
\label{qu:salazar}
Let $r$ be a rational number in $(3,4)$. Does there exist an integer $k$ and an infinite family of (simple) graphs, each of which has average degree $r$ and is $k$-crossing-critical?
\end{question}

Question \ref{qu:salazar} was partially answered by Pinontoan and Richter \cite{PR406}. They proposed constructing crossing-critical graphs from smaller pieces or tiles, and applied this idea to design infinite families of simple $k$-crossing-critical graphs for any prescribed average degree $r\in (3\frac{1}{2},4)$. 

Besides the study of degrees in crossing-critical graphs, there are also some structural results. Salazar improved the factor $\frac{5}{2}$ in the bound of Richter and Thomassen to $2$ for large $k$-crossing-critical graphs \cite{S} and for graphs of minimum degree four \cite{S297}. Hlin\v{e}n\'{y} proved that there is a function $f$ such that no $k$-cros\-sing-cri\-tical graph contains a subdivision of a binary tree of height $f(k)$, which implies that the path-width of such a graph is at most $2^{f(k)+1}-2$. In particular, $k-1\le f(k)\le 6(72\log_2 k+248)k^3$ \cite{H351,H410}. Existence of a bound on the path-width of $k$-crossing-critical graphs was first conjectured by Geelen, Richter, Salazar, and Thomas in \cite{GRS438}, where they established a result implying a bound on the tree-width of $k$-crossing-critical graphs. Hlin\v{e}n\'{y} defined crossed $k$-fences, which are $k$-crossing-critical graphs, in \cite{H351}. Crossed $k$-fences from some particular family contain subdivisions of binary trees of height $k-1$ and thus have path-width at least $2^{k}-2$. 

Focus of the research on crossing-critical graphs was on $3$-(edge)-con\-nect\-ed cros\-sing-critical graphs. This condition eliminates vertices of degree two, which are trivial with respect to the crossing number. But the condition is much stronger and its application has been justified only recently by a structural result of Lea\~{n}os and Salazar in \cite{LS}, stating that, for a connected crossing-critical graph $G$ with minimum degree at least three, there exists a collection $G_1,\ldots,G_m$ of\/ $3$-edge-connected crossing-critical graphs, each of which is contained as a subdivision in $G$, and such that $\crn(G)=\sum_{i=1}^m\crn(G_i)$.

Two new constructions of crossing-critical graphs are developed in this contribution. In combination with the one of Pinontoan and Richter \cite{PR406}, they are applied to answer the question of Salazar by a result resembling those of Šir\'{a}\v{n} and Kochol: we show that there exist infinite families of simple $k$-crossing-critical graphs with any prescribed average degree $r\in(3,6)$, for any $k$ greater than some lower bound $N_r$. This leaves average degree $r=6$ as the only open case. Several steps are required for our proof. In Section \ref{sc:tiles}, the theory of tiles of Pinontoan and Richter is extended to yield effective lower bounds on the number of tiles needed to imply the lower bounds on crossing number. Section \ref{sc:staircase} contains the first new construction of $k$-crossing critical graphs, which yields infinite families of such graphs with average degree arbitrarily close to three. The second new construction relies on a sufficient condition that the zip product, studied in \cite{CNCPP,CNCPT}, preserves criticality of the graphs involved. This is established in Section \ref{sc:zip}. The main result is proved in Section \ref{sc:families} by combining the results of the previous sections. Some further aspects of applying zip product in construction of crossing-number critical graphs are discussed as the conclusion in Section \ref{sc:aspects}.

\section{Tiles}
\label{sc:tiles}

In this section, we present a variant of the theory of tiles developed by Pinontoan and Richter \cite{PR406}. In particular, we consider general sequences of not necessarily equal tiles, avoid the condition that the tiles be connected, and allow forming double edges when joining tiles. Such generalizations do not hinder the arguments of \cite{PR406} and are useful in further investigations of tiled graphs. We establish an effective bound on the number of tiles needed to imply lower bounds on crossing numbers. Finally, we combine these improvements into a general construction of crossing-critical graphs. 

Let $G$ be a graph and $\lw=(\lw_0,\ldots,\lw_l)$, $\rho=(\rw_0,\ldots,\rw_r)$ two sequences of distinct vertices, such that no vertex of $G$ appears in both. The triple $T=(G,\lw,\rw)$ is called a \DEF{tile}. To simplify the notation, we may sometimes use $T$ in place of its graph $G$ and we may consider sequences $\lw$ and $\rw$ as sets of vertices. For $u,v\in\lw$ or $u,v\in\rw$, we use $u\le v$ or $u\ge v$ whenever $u$ precedes or succeeds $v$ in the respective sequence. 

A drawing of $G$ in the unit square $[0,1]\times[0,1]$ that meets the boundary of the square precisely in the vertices of the \DEF{left wall} $\lw$, all drawn in $\{0\}\times[0,1]$, and the \DEF{right wall} $\rw$, all drawn in $\{1\}\times[0,1]$, is a \DEF{tile drawing} of $T$ if the sequence of decreasing $y$-coordinates of the vertices of each $\lw$ and $\rw$ respects the corresponding sequence $\lw$ or $\rw$. The \DEF{tile crossing number} $\tcrn(T)$ of a tile $T$ is the minimum number of crossings over all tile drawings of $T$.

Let $T=(G,\lw,\rw)$ and $T'=(G',\lw',\rw')$ be two tiles. We say that $T$ is \DEF{compatible} with $T'$ if $|\rw|=|\lw'|$. A tile $T$ is \DEF{cyclically-compatible} if it is compatible with itself. A sequence of tiles $\cT=(T_0,\ldots,T_m)$ is \DEF{compatible} if $T_i$ is compatible with $T_{i+1}$ for $i=0,\ldots,m-1$. It is \DEF{cyclically-compatible} if it is compatible and $T_{m}$ is compatible with $T_0$. All sequences of tiles are assumed to be compatible.
\vfill
\vfill

The \DEF{join} of two compatible tiles $T$ and $T'$ is defined as $T\otimes T'=(G\otimes G',\lw,\rw')$, where $G\otimes G'$ is the graph obtained from the disjoint union of $G$ and $G'$ by identifying $\rw_i$ with $\lw'_i$ for $i=0,\ldots,|\rw|-1$. This operation is associative, thus we can define the \DEF{join} of a compatible sequence of tiles $\cT=(T_0,\ldots,T_{m})$ to be the tile $\otimes \cT=T_0\otimes T_1\otimes \ldots\otimes T_{m}$. Note that we may produce multiple edges or vertices of degree two when joining tiles. We keep the double edges, but remove the vertices of degree two by contracting one of the incident edges.
\vfill
\vfill

For a cyclically-compatible tile $T=(G,\lw,\rw)$, we define its \DEF{cyclization} $\circ T$ as the graph, obtained from $G$ by identifying $\lw_i$ with $\rw_i$ for $i=0,\ldots,|\rw|-1$. Similarly, we define the cyclization of a cyclically-compatible sequence of tiles as $\circ\cT=\circ(\otimes \cT)$.
\vfill
\vfill

\begin{lemma}[\cite{PR406}]
\label{lm:join}
Let $T$ be a cyclically-compatible tile. Then, $\crn(\circ{T})\le\tcrn(T)$. Let $\cT=(T_0,\ldots,T_{m})$ be a compatible sequence of tiles. Then, $\tcrn(\otimes \cT)\le \sum_{i=0}^{m} \tcrn(T_i)$.
\end{lemma}
\vfill
\vfill

For a sequence $\omega$, let $\bar\omega$ denote the reversed sequence. For a tile $T=(G,\lw,\rw)$, let its \DEF{right-inverted} tile $T^{\updownarrow}$ be the tile $(G,\lw,\bar\rw)$, its \DEF{left-inverted} tile $^{\updownarrow}T$ be the tile $(G,\bar\lw,\rw)$, and its \DEF{inverted} tile be the tile  $^{\updownarrow}T^{\updownarrow}=(G,\bar\lw,\bar\rw)$. The \DEF{reversed} tile of $T$ is the tile $T^{\leftrightarrow}=(G,\rw,\lw)$. 
\vfill
\vfill

Let $\cT=(T_0,\ldots,T_m)$ be a sequence of tiles. A \DEF{reversed} sequence of $\cT$ is the sequence $\cT^{\leftrightarrow}=(T_m^{\leftrightarrow},\ldots,T_0^{\leftrightarrow})$. A \DEF{twist} of $\cT$ is the sequence $\cT^{\updownarrow}=(T_0,\ldots,T_{m-1},T_m^{\updownarrow})$. Let $i\in\lset{0}{m}$ be arbitrary. Then, an $i$-\DEF{flip} of $\cT$ is the sequence $\cT^i=(T_0,\ldots,T_{i-1},T_{i}^{\updownarrow},$ ${{}^{\updownarrow}T_{i+1}},T_{i+2},\ldots,T_m)$, an $i$-\DEF{cut} of $\cT$ is the sequence $\cT/i=(T_{i+1},\ldots,T_m,T_0,\ldots,T_{i-1})$, and an $i$-\DEF{shift} of $\cT$ is the sequence $\cT_i=(T_i,\ldots,T_m,T_0,\ldots,T_{i+1})$. For the last two operations, cyclic compatibility of $\cT$ is required.
\vfill
\vfill

Two sequences of tiles $\cT$ and $\cT'$ of the same length $m$ are \DEF{equivalent} if one can be obtained from the other by a sequence of shifts, flips, and reversals. It is easy to see that the graphs $\circ \cT$ and $\circ \cT'$ are equal for equivalent cyclically-compatible sequences $\cT$ and $\cT'$ and thus have the same crossing number. 

We say that a tile $T=(G,\lw,\rw)$ is \DEF{planar} if $\tcrn(T)=0$ holds. It is \DEF{connected} if $G$ is connected. It is \DEF{perfect} if:
\begin{enumerate}[label=(p.\roman{*}), ref=(p.\roman{*})]
\item \label{pt:selfCompatible} $|\lw|=|\rw|$, 
\item \label{pt:gmwConnected} both graphs $G-\lw$ and $G-\rw$ are connected, 
\item \label{pt:pathNotInWall} for every $v\in \lw$ or $v\in \rw$ there is a path from $v$ to a vertex in $\rw$ ($\lw$) in $G$ internally disjoint from $\lw$ ($\rw$), and
\item \label{pt:pairedPath} for every $0 \le i < j \le |\lw|$ there is a pair of disjoint paths $P_{ij}$ and $P_{ji}$ in $G$, such that $P_{ij}$ joins $\lw_i$ with $\rw_i$ and $P_{ji}$ joins $\lw_j$ with $\rw_j$. 
\end{enumerate}

\noindent Note that perfect tiles are connected.

\begin{lemma}[\cite{PR406}]
\label{lm:boundExists}
For a cyclically-compatible perfect planar tile $T$ and a compatible sequence $\cT=(T_0,\ldots,T_m,T)$, there exists $n\in \NN$, such that, for every $k\ge n$, $\tcrn((\otimes\cT)\otimes (T^k))=\tcrn((\otimes\cT)\otimes (T^n))$.
\end{lemma}

Let $T=(G,\lw,\rw)$ be a tile and $H$ a graph that contains $G$ as a subgraph. The \DEF{complement} of the tile $T$ in $H$ is the tile $H-T=(H[(V(H)\setminus V(G))\cup\lw\cup\rw]-E(G),\rw,\lw)$. We can consider it as the edge complement of the subgraph $G$ of $H$ from which we remove all the vertices of $T$ not in its walls. Whenever $\circ(T\otimes (H-T))=H$, i.e.\ if the vertices of $\lw\cup\rw$ separate $G$ from $H - G$, we say that $T$ is a \DEF{tile in $H$}. Using this concept, the following lemma shows the essence of perfect tiles.

\begin{lemma}
\label{lm:perfect}
Let $T=(G,\lw,\rw)$ be a perfect planar tile in a graph $H$, such that there exist two disjoint connected subgraphs $G_\lw$ and $G_\rw$ of $H$ contained in the same component of $H-T$ and with $G\cap G_\lw=(\lw,\emptyset)$, $G\cap G_\rw=(\rw,\emptyset)$. If $E(G)$ and either $E(G_\lw)$ or $E(G_\rw)$ are not crossed in some drawing $D$ of $H$, then the $D$-induced drawings of\/ $T$ and its complement $H-T$ are homeomorphic to tile drawings.
\end{lemma}
\begin{proof}
There is only one component of $H-T$ containing the vertices of $\lw\cup\rw$, and as the edges of other components do not cross $G$ nor influence its induced drawing, we may assume that $H-T$ {\em is} that component and, in particular, it is connected.

Denote by $D_T$ the $D$-induced drawing of $T$, by $T^-$ the tile $H-T$, and by $D^-$ the $D$-induced drawing of $T^-$. As the edges of $T$ are not crossed in $D$ and $T^-$ is connected, there is a face $F$ of $D_T$ containing $D^-$. The boundary of $F$ contains all vertices of $T\cap T^-=\lw\cup\rw$. Let $W$ be the facial walk of $F$. No vertex of $\lw\cup \rw$ appears twice in $W$: such a vertex would be a cutvertex in the planar graph $G$. Then either $G-\lw$ or $G-\rw$ would not be connected, violating \ref{pt:gmwConnected}, or some vertex in $\lw\cup\rw$ would have no path to the opposite wall, as required by \ref{pt:pathNotInWall}. 

Let $W'$ be the induced sequence of vertices of $\lw\cup \rw$ in $W$. As the edges of $G_\lw$ or $G_\rw$ are not crossed in $D$ and $T$, $G_\lw$, and $G_\rw$ are connected, the vertices of $\lw$ do not interlace with the vertices of $\rw$ in $W'$. The ordering of $\lw$ in $W'$ is the inverse ordering of $\rw$ in $W'$, since the disjoint paths from \ref{pt:pairedPath} do not cross in $D_T$. The planarity and the connectedness of $T$ imply that whenever $i<j<l$ or $i>j>l$, there is a path $Q$ from $P_{jl}$ to $\lw_i$ disjoint from $P_{lj}$. $Q$ does not cross $P_{lj}$ in $D_T$, thus $W'=\lw\bar\rw$ or $W'=\rw\bar\lw$. The claim follows.
\end{proof}

The above arguments were in \cite{PR406} combined with Lemma \ref{lm:boundExists} to demonstrate the following:

\begin{theorem}[\cite{PR406}]
Let $T$ be a perfect planar tile and let $\bar T_k=T^k\otimes T^{\updownarrow}\otimes T^k$ for $k\ge 1$. Then there exist integers $n,N$, such that $\crn(\circ(\bar T_k))=\tcrn(\bar T_n)$ for every $k\ge N$.
\end{theorem}

\noindent We establish effective values of $n$ and $N$ from the above theorem:

\begin{theorem}
\label{th:generalTilesIe}
Let $\cT=(T_0,\ldots,T_l,\ldots,T_m)$ be a cyclically-compatible sequence of tiles. Assume that, for some integer $k\ge 0$, the following hold: $m\ge 4k-2$, $\tcrn(\otimes \cT/i)\ge k$, and the tile $T_i$ is a perfect planar tile, both for every $i=0,\ldots,m$, $i\neq l$. Then, $\crn(\circ\cT)\ge k$. 
\end{theorem}
\begin{proof}
We may assume $k\ge 1$. Let $G=\circ\cT$ and let $D$ be an optimal drawing of $G$. Assume that $D$ has less than $k$ crossings. Then there are at most $2k-1$ tiles in the set $\cS=\dset{T_i} {i=l\hbox{ or } E(T_i)\hbox{ cros\-s\-ed in }D}$. The circular sequence $\cT$ is by the tiles of $\cS$ fragmented into at most $2k-1$ segments. By the pigeon-hole principle, the set $\cT\setminus \cS$, which consists of at least $2k$ tiles, contains two consecutive tiles $T_iT_{i+1}$. Assume for simplicity that $i=1$, then either $T_{0}$ or $T_{3}$ is distinct from $T_l$. Lemma \ref{lm:perfect} with $(G,T_1,T_{0},T_{2})$ or $(G,T_{2},T_{1},T_{3})$ in place of $(H,T,G_\lw,G_\rw)$ establishes that the induced drawing $D^-$ of $G-T_j$ is a tile drawing for some $j\in\{1,2\}$. Since $D^-$ contains all the crossings of $D$, this contradicts $\tcrn(\otimes(\cT/j))\ge k$, and the claim follows.
\end{proof}

\begin{corollary}
\label{cr:generalTiles}
Let $\cT=(T_0,\ldots,T_l,\ldots,T_m)$ be a cyclically-compatible sequence of tiles and $k=\min_{i\neq l}\tcrn(\otimes \cT/i)$. If $m\ge 4k-2$ and the tile $T_i$ is a perfect planar tile for every $i=0,\ldots,m$, $i\neq l$, then $\crn(\circ\cT)=k$. 
\end{corollary}
\begin{proof}
By Lemma \ref{lm:join} and the planarity of tiles, $\crn(\circ\cT) \le \tcrn((\otimes \cT/i)\otimes T_i)\le \tcrn(\otimes \cT/i)$ for any $i\neq l$, thus $\crn(\circ\cT)\le k$. Theorem \ref{th:generalTilesIe} establishes $k$ as a lower bound and the claim follows.
\end{proof}

A tile $T$ is $k$-\DEF{degenerate} if it is perfect, planar, and $\tcrn(T^{\updownarrow}-e)<k$ for any edge $e\in E(T)$. A sequence of tiles $\cT=(T_0,\ldots,T_m)$ is $k$-\DEF{critical} if the tile $T_i$ is $k$-degenerate for every $i=0,\ldots,m$ and $\min_{i\neq m}\tcrn(\otimes(\cT^{\updownarrow}/i))\ge k$. Note that $\tcrn(T^{\updownarrow})\ge k$ for every tile $T$ in a $k$-critical sequence.

\begin{corollary}
\label{cr:criticalTiles}
Let $\cT=(T_0,\ldots,T_m)$ be a $k$-critical sequence of tiles. Then, $T=\otimes \cT$ is a $k$-degenerate tile. If $m\ge 4k-2$ and $\cT$ is cyclically-compatible, then $\circ(T^{\updownarrow})$ is a $k$-crossing-critical graph.
\end{corollary}
\begin{proof}
Lemma \ref{lm:join} implies that $T$ is a planar tile. By induction it is easy to show that $T$ is a perfect tile. Let $e$ be an edge of $T$ and let $i$ be such that $e\in T_i$. The sequence $\cT'=(T_{0},\ldots,T_{i-1},T_{i}^{\updownarrow},{{{}^{\updownarrow}}T_{i+1}}^{\updownarrow},\ldots,{{{}^{\updownarrow}T_{m}^{\updownarrow}}})$ is equivalent to $\cT^{\updownarrow}$. Lemma \ref{lm:join} establishes $\tcrn(T^{\updownarrow}-e)=\tcrn((\otimes \cT')-e)\le \tcrn(T_{i}^\updownarrow-e)<k$, thus $T$ is a $k$-degenerate tile. 

Let $\cT$ be cyclically-compatible. Then $\crn((\circ T^{\updownarrow})-e)<k$ for any edge $e\in E(T)$. Theorem \ref{th:generalTilesIe} implies $\crn(\circ (T^{\updownarrow}))\ge k$ for $m\ge 4k-2$. Thus, $\circ (T^{\updownarrow})$ is a $k$-crossing-critical graph. 
\end{proof}

The above results provide sufficient conditions for the crossing numbers of certain graphs to be estimated in terms of the tile crossing numbers of their subgraphs. In what follows, we develop some techniques to estimate the tile crossing number. 

A general tool we employ for this purpose is the concept of a \DEF{gadget}. We do not define it formally; a gadget can be any structure inside a tile $T=(G,\lw,\rw)$, which guarantees a certain number of crossings in every tile drawing of $T$. Pinontoan and Richter used twisted pairs as gadgets \cite{PR406}, and we present staircase strips. Some other possible gadgets are cloned vertices, which were already used by Kochol \cite{K119}, wheel gadgets, and others, which were studied in \cite{SACN}. 

In general, there can be many gadgets inside a single tile. Whenever they are edge disjoint, the crossings they force in tile drawings are distinct. The following weakening of disjointness enables us to prove stronger results. For clarity, we first state the condition in its set-theoretic form.

Let $A_1,B_1,A_2,B_2$ be four sets. The unordered pairs $\{A_1,B_1\}$ and $\{A_2,B_2\}$ are 
\DEF{coherent} if one of the sets $X_i$, $X\in\{A,B\}$, $i\in \{1,2\}$, is disjoint from $A_{3-i}\cup B_{3-i}$.

\begin{lemma}
\label{lm:coherent}
Let $\{A,B\}$ and $\{A',B'\}$ be two pairs of sets. If they are coherent and 
\begin{equation}
\label{eq:coherent}
a\in A, b\in B, a'\in A'\hbox{ and }b'\in B',
\end{equation}
then the unordered pairs $\{a,b\}$ and $\{a',b'\}$ are distinct. Conversely, if (\ref{eq:coherent}) implies distinctness of $\{a,b\}$, $\{a',b'\}$ for every quadruple $a,b,a',b'$, then the pairs $\{A,B\}$, $\{A',B'\}$ are coherent.
\end{lemma}
\begin{proof}
Suppose the pairs are not distinct, then either $a=a'$ and $b=b'$, or $a=b'$ and $b=a'$. In both cases, every set has a member in the union of the other pair, and the pairs are not coherent.

For the converse, suppose the pairs would not be coherent. Then every set would contain an element in the union of the opposite pair. Let $x\in A\cap A'$, assuming the intersection is not empty. If there is an element $y\in B'\cap B$, then the quadruple $a=x$, $b=y$, $a'=x$, $b'=y$ satisfies (\ref{eq:coherent}) but does not form two distinct pairs. If $B\cap B'$ is empty, then there must be $a'\in B\cap A'$ and $b'\in B'\cap A$. The quadruple $a=a'$, $b=b'$, $a'$, $b'$ satisfies (\ref{eq:coherent}). Assuming $x\in A\cap B'$, a similar analysis applies and the claim follows.
\end{proof}

Lemma \ref{lm:coherent} has an immediate application to crossings: whenever the pairs of edges $\{e_x,f_x\}$ and $\{e_y,f_y\}$ are distinct for two crossings $x$ and $y$, the crossings $x$ and $y$ are distinct. Distinctness of crossings induced by two coherent pairs of sets of edges in a graph follows.

The notion of coherence can be generalized. Let $\lset{A_1}{A_m}$ and $\{B_1,\ldots,$ $B_n\}$ be two families of sets. They are \DEF{coherent} if the two pairs $\{A_i,A_j\}$ and $\{B_k,B_l\}$ are coherent for every $0\le i<j\le m$, $0\le k<l\le n$.

A path $P$ in $G$ is a \DEF{traversing path} in a tile $T=(G,\lw,\rw)$ if there exist indices $i(P)\in\{0,\ldots,|\lw|-1\}$ and $j(P)\in\{0,\ldots,|\rw|-1\}$ such that $P$ is a path from $\lw(P)=\lw_{i(P)}$ to $\rw(P)=\rw_{j(P)}$ and $\lw(P)$, $\rw(P)$ are the only wall vertices that lie on $P$. An (unordered) pair of disjoint traversing paths $\{P,Q\}$ is \DEF{aligned} if $i(P)<i(Q)$ is equivalent to $j(P)<j(Q)$, and \DEF{twisted} otherwise. Disjointness of the traversing paths in a twisted pair $\{P,Q\}$ implies that some edge of $P$ must cross some edge of $Q$ in any tile drawing of $T$. Two pairs $\{P,Q\}$ and $\{P',Q'\}$  of traversing paths in $T$ are \DEF{coherent} if $\{E(P),E(Q)\}$ and $\{E(P'),E(Q')\}$ are coherent. A family of pairwise coherent twisted (respectively, aligned) pairs of traversing paths in a tile $T$ is called a \DEF{twisted} (\DEF{aligned}) family in $T$.

\begin{lemma}[\cite{PR406}]
\label{lm:twistedbound}
Let $\cF$ be a twisted family in a tile $T$. Then, $\tcrn(T)\ge |\cF|$.
\end{lemma}

Let a tile $T$ be compatible with $T'$ and let $\{P,Q\}$ be a twisted pair of traversing paths of $T$. An aligned pair $\{P',Q'\}$ of traversing paths in $T'$ \DEF{extends} $\{P,Q\}$ to the right if $j(P)=i(P')$, $j(Q)=i(Q')$. Then $\{PP',QQ'\}$ is a twisted pair in $T\otimes T'$. For a twisted family $\cF$ in $T$, a \DEF{right-extending family} is an aligned family $\cF'$ in $T'$, for which there exists a bijection $\dfnc{e}{\cF}{\cF'}$, such that the pair $e(\{P,Q\})\in\cF'$ extends the pair $\{P,Q\}$ on the right. In this case, the family $\cF\otimes_e \cF'=\dset{\{PP',QQ'\}}{\{P',Q'\}=e(\{P,Q\})}$ is a twisted family in $T\otimes T'$. Extending to the left is defined similarly. Let $\cT=(T_0,\ldots,T_l,\ldots,T_m)$ be a compatible sequence of tiles and $\cF_l$ a twisted family in $T_l$. If, for $i=l+1,\ldots,m$ (respectively, $i=l-1,\ldots,0$), there exist aligned right- (left-) extending families $\cF_i$ of $\cF_l\otimes\ldots\otimes \cF_{i-1}$ ($\cF_{i+1}\otimes\ldots\otimes \cF_{l-1}$), then $\cF_l$  \DEF{propagates} to the right (left) in $\cT$. $\cF_l$ \DEF{propagates} in cyclically-compatible $\cT$ if it propagates both to the left and to the right in every cut $\cT/i$, $i=0,\ldots,m$, $i\neq l$.

A twisted family $\cF$ in a tile $T$ \DEF{saturates} $T$ if $\tcrn(T)=|\cF|$, i.e. there exists a tile drawing of $T$ with $|\cF|$ crossings. Clearly, all these crossings must be on the edges of pairs of paths in $\cF$. 

\begin{corollary}
\label{cr:saturated}
Let $\cT=(T_0,\ldots,T_l,\ldots,T_m)$ be a cyclically-compatible sequence of tiles and $\cF$ a twisted family in $T_l$ that propagates in $\cT$. If $m\ge 4|\cF|-2$ and the tile $T_i$ is a perfect planar tile for every $i=0,\ldots,m$, $i\neq l$, then $\crn(\circ\cT)\ge |\cF|$. If $\cF$ saturates $T_l$, then the equality holds.
\end{corollary}
\begin{proof}
As $\cF$ propagates in $\cT$, Lemma \ref{lm:twistedbound} implies $\min_{i\neq l}\tcrn(\otimes (\cT/i))\ge|\cF|$.  Theorem \ref{th:generalTilesIe} establishes the claim.
\end{proof}

\begin{figure}
\psfragscanon
\psfrag{AA}{(a)}
\psfrag{BB}{(b)}
\psfrag{A}{$A$}
\psfrag{B}{$B$}
\psfrag{C}{$C$}
\psfrag{D}{$D$}
\psfrag{E}{$E$}
\psfrag{F}{$F$}
\psfrag{G}{$G$}
\psfrag{P}{$P_i$}
\psfrag{P1}{$P_1$}
\psfrag{Pi}{$P_{i+1}$}
\psfrag{Q}{$Q_i$}
\psfrag{R}{$R_i$}
\psfrag{S}{$S_i$}
\psfrag{Sw}{$S_{2w+1}$}
\psfrag{Sub}{\small $2w+1$ subtiles}
\begin{center}
\includegraphics[height=68mm]{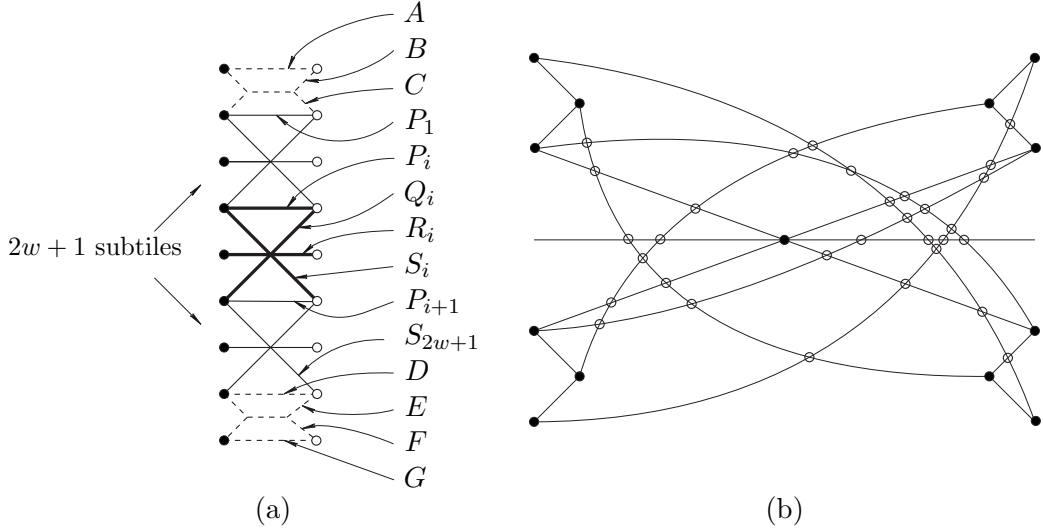}
\end{center}
\caption{(a) The tile $H_w$, $w=1$. (b) An optimal tile drawing of $H_0$.}
\label{fg:graph6}
\end{figure}

Let $H_w$ be a tile, which is for $w=1$ presented in Figure \ref{fg:graph6} (a). It is constructed by joining two subtiles, denoted by dashed edges, with a sequence of $2w+1$ subtiles, of which one is drawn with thick edges. The left (right) wall vertices of $H_w$ are colored black (white). $H_w$ is a perfect planar tile. Let $\cH(w,s)=(H_w,\ldots,H_w)$ be a sequence of tiles of length $s$ and let $H(w,s)=\circ(\cH(w,s)^{\updownarrow})$ be the cyclization of its twist. 

\begin{proposition}
\label{pr:graphsH}
The graph $H(w,s)$ is a crossing-critical graph with crossing number $k=32 w^2+56w+ 31$ whenever $s\ge 4k-1$.
\end{proposition}
\begin{proof}
Using the traversing paths $A,\ldots,G$ of $H_w$ depicted in Figure \ref{fg:graph6}, we construct an aligned family $\cH_w$ of size $k$, cf.~\cite{SACN,MAT}. The corresponding family $\cH_w'$ in $H_w^{\updownarrow}$ is twisted and propagates in $\cH(w,s)^{\updownarrow}$. Figure \ref{fg:graph6} (b) presents an optimal tile drawing of $H_0$, its generalization to $w>0$ demonstrates that $\cH_w'$ saturates $H_w^{\updownarrow}$. The crossing number of $H(w,s)$ is established by Corollary \ref{cr:saturated}. 

The number of crossings can be decreased after removing any edge from the drawing in Figure \ref{fg:graph6} (b). This also applies to the generalization of the drawing, thus $H_w$ is a $k$-degenerate tile. The propagation of the twisted family $\cF'$  demonstrates $\tcrn\left(\otimes(\cH(w,s)^{\updownarrow}/i)\right)\ge k$ for any $i\ne s$, thus $\cH(w,s)$ is a $k$-critical sequence. Criticality of $H(w,s)$ follows by Corollary \ref{cr:criticalTiles}.
\end{proof}

\section{Staircase strips in tiles}
\label{sc:staircase}

In this section, we study twisted staircase strips. Using these gadgets, we construct new crossing-critical graphs with average degree close to three. 

\begin{figure}
\psfragscanon
\psfrag{S1}{$u=u_1=u_1'=u_2=v_1$}
\psfrag{T}{$v=u_n'=v_{n-1}'=v_n=v_n'$}
\psfrag{S2}{$s$}
\psfrag{S3}{$s'$}
\psfrag{S4}{$s''$}
\psfrag{A}{$u_2'$}
\psfrag{B}{$u_3=u_3'$}
\psfrag{C}{$u_5$}
\psfrag{D}{$u_5'$}
\psfrag{E}{$u_n$}
\psfrag{H}{$v_1'$}
\psfrag{I}{$v_2$}
\psfrag{F}{$S_2$}
\psfrag{G}{$v_{n-1}$}
\begin{center}
\includegraphics[width=135mm]{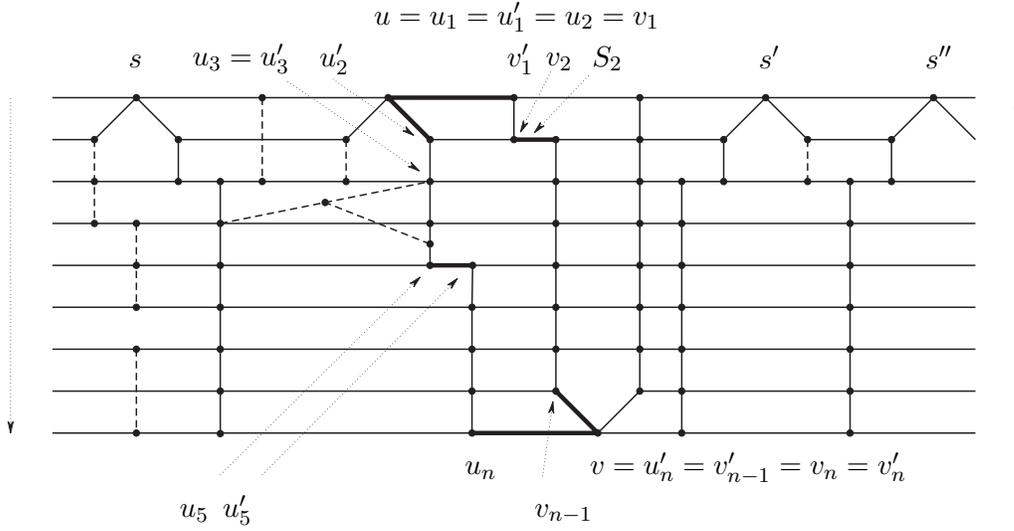}
\end{center}
\caption{A general staircase strip in a tile. Leftmost and rightmost arrows indicate the ordering of the wall vertices. Dashed edges are part of the tile but not of the staircase strip.}
\label{fg:generalStaircase}
\end{figure}

Let $\cP=\lset{P_1,P_2}{P_n}$ be a sequence of traversing paths in a tile $T$ with the property $\lw(P_i)\le\lw(P_j)$ and $\rw(P_i)\ge\rw(P_j)$ for $i<j$. Assume that they are pairwise disjoint, except for the pairs $P_1,P_2$ and $P_{n-1},P_{n}$, which may share vertices, but not edges. For $u\in V(P_1)\cap V(P_2)$ and $v\in V(P_{n-1})\cap V(P_n)$, we say that $u$ is \DEF{left} of $v$ (cf. Figure \ref{fg:generalStaircase}) if there exist internally disjoint paths $Q_u$ and $Q_v$ from $u$ to $v$ such that:  
\begin{enumerate}[label=(s.\roman{*}), ref=(s.\roman{*})]
\item\label{pt:ui} there exist vertices $u_1,u_1',\ldots,u_n,u_n'$ that appear in this order on $Q_u$,
\item\label{pt:vi} there exist vertices $v_1,v_1',\ldots,v_{n},v_{n}'$ that appear in this order on $Q_v$,
\item\label{pt:uv} $u=u_1=u_1'=u_2=v_1$ and $v=u_n'=v_{n-1}'=v_n=v_n'$, 
\item\label{pt:notIn} $v_1',v_2,v_2',u_2'\not\in P_1\cap P_2$ and $v_{n-1},u_{n-1},u_{n-1}',u_n\not \in P_{n-1}\cap P_n$,
\item\label{pt:Ri} for $i=1,\ldots,n$, $R_i:=u_iP_iu_i'\subseteq P_i\cap Q_u$, with equality for $i\neq n-1$,
\item\label{pt:Si} for $i=1,\ldots,n$, $S_i:=v_iP_iv_i'\subseteq P_i\cap Q_v$, with equality for $i\neq 2$,
\item\label{pt:Rn1S2} $R_{n-1}=(P_{n-1}\cap Q_u)-R_n$ and $S_2=(P_2\cap Q_v)-S_1$,
\item\label{pt:v1v2} if $'u,u'\in P_1\cap P_2$ are two vertices with $v_1'\in {'u}P_1u'$, then $v_2\in {'u}P_2u'$, 
\item\label{pt:un1un} if $'v,v'\in P_{n-1}\cap P_{n}$ are two vertices with $u_n\in {}{'v}P_nv'$, then $u_{n-1}'\in {'v}P_{n-1}v'$, and 
\item\label{pt:order} $\lw(P_i)u_iu_i'v_iv_i'\rw(P_i)$ lie in this order on $P_i$ for $i=1,\ldots,n$.
\end{enumerate}
Similarly, we define when $u$ is \DEF{right} of $v$. We say that $\cP$ forms a \DEF{twisted staircase strip} of width $n$ in the tile $T$ if the vertex $u$ is either left or right of the vertex $v$ whenever $u\in V(P_1)\cap V(P_2)$ and $v\in V(P_{n-1})\cap V(P_n)$. 

Vertex $u$ in Figure \ref{fg:generalStaircase} is left of $v$. The features establishing this fact are emphasized. The subpaths $u_i'Q_uu_{i+1}$ and $v_i'Q_vv_{i+1}$ are, for $i=2,\ldots,n-1$, internally disjoint from $P_j$ by \ref{pt:Ri} and \ref{pt:Si}, for any $j=1,\ldots,n$, and their length is at least one. They are represented by solid vertical edges in the figure. However, the length of $R_i$ and $S_i$, $i=1,\ldots,n$, may be zero; the thick edges in the figure emphasize the instances when their length is positive. Solid edges in Figure \ref{fg:generalStaircase} are part of a twisted staircase strip, dashed edges are not. Note that the vertices $u$ and $s$ are left of $v$ and that the vertices $s'$ and $s''$ are right of $v$.

\begin{theorem}
\label{th:tcrnStrip}
Let $T$ be a tile and assume that $\cP=\lset{P_1,P_2}{P_n}$ forms a twisted staircase strip of width $n$ in $T$. Then, $\tcrn(T)\ge {n\choose 2}-1$.
\end{theorem}
\begin{proof}
If a wall vertex $v$ in a tile $T$ has degree $d$, then the tile crossing number of $T$ is not changed if $d$ new neighbors $v_1,\ldots,v_d$ of degree one are attached to $v$ and $v$ is in its wall replaced by $v_1,\ldots,v_d$. Thus, we may assume that all paths in $\cP$ have distinct startvertices in $\lw$ and distinct endvertices in $\rw$.

Let $D$ be any optimal tile drawing of $T$. By Lemma \ref{lm:twistedbound}, there are at least ${n\choose 2}-2$ crossings in $D$, since the set $\cF = \dset{\{P_i,P_j\}}{1\le i<j\le n} \setminus \{\{P_1,P_2\},\{P_{n-1},P_n\}\} $ is a twisted family in $T$. For $\{P_i,P_j\}\in\cF$, let $P_i$ cross $P_j$ at $x_{i,j}$. In what follows, we contradict the assumption 
\begin{equation}
\label{eq:assumption}
\hbox{\textsl{$x_{i,j}$ are all the crossings of $D$}}.
\end{equation}

For $i=1,\ldots,n$, let $P_i$ be oriented from $\lw(P_i)$ to $\rw(P_i)$. The assumption (\ref{eq:assumption}) implies that the induced drawing of every $P_i$ is a simple curve. This curve splits the unit square $\Delta=I\times I$ containing $D$ into two disjoint open disks, the lower disk $\Delta_i^-$ bordering $[0,1]\times \{0\}$ and the upper disk $\Delta_i^+$ bordering $[0,1]\times \{1\}$.

\textbf{Claim 1:} \textsl{At $x_{i,j}$, the path $P_j$ crosses from $\Delta_i^-$ into $\Delta_i^+$ and the path $P_i$ crosses from $\Delta_j^+$ into $\Delta_j^-$.} This follows from $i<j$ and the orientation of paths $P_i$ and $P_j$.

As $\lw(P_2)\in\Delta_1^-$ and $\rw(P_2)\in\Delta_1^+$, there is a vertex $u\in V(P_1)\cap V(P_2)$ where $P_2$ crosses $P_1$ from $\Delta_1^-$ to $\Delta_1^+$. Also, there is a vertex $v\in V(P_{n-1})\cap V(P_n)$, such that $P_{n-1}$ crosses from $\Delta_n^+$ into $\Delta_n^-$ at $v$. Then Claim 1 holds for $x_{1,2}=u$ and $x_{n-1,n}=v$. 

By symmetry, we may assume that $u$ is left of $v$ in $T$. Let $Q_u$ and $Q_v$ be the corresponding paths in $T$. $P_2$ enters $\Delta_1^+$ at $u$, and (\ref{eq:assumption}), \ref{pt:uv}, \ref{pt:notIn}, and \ref{pt:Ri} imply $u_2'\in\Delta_1^+$. Similarly, $v_{n-1}\in\Delta_n^+$ by (\ref{eq:assumption}), \ref{pt:uv}, \ref{pt:notIn}, and \ref{pt:Si}. 

\textbf{Claim 2:} \textsl{If any point $y$ of $w_i'Q_w$ lies in $\Delta_i^-$ for $w\in\{u,v\}$ and $i\in\{1,\ldots,n\}$, $w_i\neq u_{n-1}$, then the path $Q_w$ must at $w_i'$ enter $\Delta_i^-$.} If $w\neq u$ or $i\neq n-1$, the segment $w_i'Q_wy$ does not cross from $\Delta_i^+$ to $\Delta_i^-$ due to Claim 1, thus it must lie in $\Delta_1^-$. 

\textbf{Claim 3:} \textsl{If there is a point $y$ of $Q_ww_i$ in $\Delta_i^-$ for $w\in\{u,v\}$ and $i\in\{1,\ldots,n\}$, $w_i\neq v_2$, then $Q_w$ must at $w_i$ leave $\Delta_i^-$.} Otherwise, the segment $yQ_ww_i$ would contradict Claim 1 at $x_{ji}$ for some $j<i$.

\textbf{Claim 4:} \textsl{For $3\le i\le n$, neither of $u_i,v_i$ lies in $\Delta_1^-$.} Assume some $u_i\in\Delta_1^-$. As $u_2'\in\Delta_1^+$, the path $Q_u$ would contradict Claim 1 at $x_{1,j}$ for some $j$, $1<j<i$. Assume $v_i\in\Delta_1^-$. Due to the orientation of $P_i$, (\ref{eq:assumption}), and \ref{pt:order}, $u_i\in\Delta_1^-$, a contradiction.

\textbf{Claim 5:} \textsl{For $1\le i\le n-2$, neither of $u_i', v_i'$ lies in $\Delta_n^-$.} Assume $v_i'\in\Delta_n^-$. As $v_{n-1}\in\Delta_n^+$, the path $Q_v$ would contradict Claim 1 at $x_{j,n}$ for some $j$, $i<j<n$. To complete the proof, observe that if $u_i'\in\Delta_n^-$ then $v_i'\in\Delta_n^-$ by (\ref{eq:assumption}) and \ref{pt:order}.

In what follows, we prove that the subdrawing of $D$ induced by $Q_u\cup Q_v\cup\left(\bigcup_i P_i\right)$ contains a \textsl{new} crossing, distinct from $x_{i,j}$, which contradicts (\ref{eq:assumption}). We first simplify the subdrawing and obtain a drawing $D'$ in which for every $i,j$, $1\le i<j\le n$, the paths $P_i$ and $P_j$ share precisely one point. We use the following steps:
\begin{itemize}
\item All vertices of $P_1\cap P_2$, $P_{n-1}\cap P_n$ at which the two paths do not cross are split.

\item As $D$ is a tile drawing, there is an even number of crossing vertices in $V(P_1)\cap V(P_2)$ preceding $u$ on $P_1$. For a consecutive pair $x,y$ of such vertices, the paths $P_1$ and $P_2$ are uncrossed by rerouting $xP_1y$ along $xP_2y$ and vice versa. The vertices $x$ and $y$ are split afterwards. The segments of $P_{n-1}$ and $P_n$ following $v$ are uncrossed in a similar manner. By \ref{pt:ui}, \ref{pt:vi}, \ref{pt:uv}, and \ref{pt:order}, the paths $Q_u$ and $Q_v$ are not affected.

\item For any pair of vertices of $S_1\cap P_2-\{u\}$, the paths $P_1$ and $P_2$ are uncrossed in the same way. Due to \ref{pt:Ri}, the vertex $u_2'$ is not on any of the two affected segments. Due to \ref{pt:notIn}, \ref{pt:Si} and \ref{pt:v1v2}, neither of the segments can contain $v_1'$, $v_2$ or $v_2'$. Thus, $u_2',v_2,v_2'\in P_2$ and $v_1'\in P_1$ after the uncrossing. As all the pairs can be uncrossed, we may assume there is at most one crossing vertex in $S_1\cap P_2$ distinct from $u$. But existence of such vertex implies by \ref{pt:v1v2} that $v_2\in\Delta_1^-$, further implying by \ref{pt:Si} and \ref{pt:Rn1S2} that $v_2'\in\Delta_1^-$. By (\ref{eq:assumption}), the segment $v_2'Q_vv_3$ does not cross $P_1$, thus $v_3$ lies in $\Delta_1^-$, contradicting Claim~4.

\item As in the previous step, the paths $P_{n-1}$ and $P_n$ are uncrossed at any pair of vertices of $R_n\cap P_{n-1}$. Existence of a single remaining crossing vertex in $R_n\cap P_{n-1}$ would by \ref{pt:notIn}, \ref{pt:Ri}, \ref{pt:Rn1S2}, and \ref{pt:un1un} imply $u_{n-2}'\in\Delta_n^-$, violating Claim~5.

\item As $D'$ is a tile drawing, there is an even number of crossing vertices in $v_1'P_1\cap P_2$. By \ref{pt:v1v2} and \ref{pt:order}, uncrossing the paths $P_1$, $P_2$ as before does not affect $Q_v$. Similarly, uncrossing the paths $P_{n-1}$ and $P_nu_n$ does not affect $Q_u$ due to \ref{pt:un1un} and \ref{pt:order}.
\end{itemize}

All crossings in thus obtained drawing $D'$ are also crossings of $D$, but some crossings of $P_1$ with $P_i$ may have become crossings of $P_2$ and $P_i$ and vice versa. The same applies to the pair $(P_{n-1},P_n)$. We replace the labels $x_{i,j}$ accordingly. Until the end of the proof, we are concerned with the drawing $D'$ only. In the new drawing, Claim 2 holds for $w_i=u_{n-1}$, Claim 3 for $w_i=v_2$, Claim 4 for $i=2$, and Claim 5 for $i=n-1$.

\textbf{Claim 6:} \textsl{For $1\le i<j\le n$, the subpath $R_i$ of $Q_u$ does not cross the subpath $S_j$ of $Q_v$ at $x_{i,j}$.} Suppose it does and take the maximal such $i$. By Claim 1 and \ref{pt:order}, $u_j$ and $v_j$ lie in $\Delta_i^-$. Claim 2 implies that $Q_u$ and $Q_v$ enter $\Delta_i^-$ at $u_i'$ and $v_i'$. Similarly, $u_i'$ and $v_i'$ lie in $\Delta_j^-$ and Claim 3 implies that $Q_u$ and $Q_v$ leave $\Delta_j^-$ at $u_j$, $v_j$. Thus, the segments $u_i'Q_uu_j$ and $v_i'Q_vv_j$ lie in the intersection $\Delta'=\Delta_i^-\cap\Delta_j^-$. $\Delta'$ is a disk as $P_i$ and $P_j$ do not self-cross and cross each other only once. The vertices $u_i'$, $v_i'$, $u_j$, $v_j$ lie in this order on the boundary of $\Delta'$, so the segments must intersect in $\Delta'$. This contradicts either the assumption (\ref{eq:assumption}) or the maximality of $i$. Claim 6 follows.

Let $\gamma^u$ denote the simplified path $P_1uQ_uvP_{n-1}$: whenever this path self-crosses, the circuit is shortcut. Let  $\gamma_1^u$, $\gamma_2^u$, and $\gamma_3^u$ be the (possibly empty) segments of $\gamma^u$ corresponding to $P_1$, $Q_u$, and $P_{n-1}$. Similarly, let $\gamma^v$ denote the simplified path $P_2uQ_vvP_n$ with the segments $\gamma_1^v$, $\gamma_2^v$, and $\gamma_3^v$. Using the induced orientation of $\gamma^u$ and $\gamma^v$, we define disks $\Delta_u^+$, $\Delta_u^-$, $\Delta_v^+$, and $\Delta_v^-$ to be the respective lower and upper disks. The endvertices of $\gamma^u$ and $\gamma^v$ interlace in the boundary of $[0,1]\times [0,1]$, thus these paths must cross at some crossing $z=z_{i,j}$ of segments $\gamma_i^u$ and $\gamma_j^v$. We contradict the assumption that $z=x_{i,j}$ for some $i,j$. Due to the definition of $\gamma^u$ and $\gamma^v$, there are nine possibilities for $z$:

\begin{enumerate}[label=(\arabic{*})]
\item $z=z_{1,1}=u$ is a touching of $\gamma^u$ and $\gamma^v$.
\item $z=z_{1,2}=x_{1,i}$ for some $i>2$. Thus, $v_i\in\Delta_1^-$ contradicts Claim 4.
\item $z=z_{1,3}=x_{1,n}$ implies $u_n\in\Delta_1^-$. 
\item $z=z_{2,1}=x_{i,2}$  for some $i>2$, then $u_i\in\Delta_1^-$.
\item $z=z_{2,2}=x_{i,j}$ is a crossing of $S_i$ and $R_j$. Claim 6 implies that $1\le i<j\le n$. Choose smallest such $i$ and then smallest $j$. $Q_v$ starts in $\Delta_u^-$ and since $z$ is the first crossing of $Q_v$ with $\gamma^u$ (or one of the other eight cases would apply), $Q_v$ leaves $\Delta_u^-$ and enters $\Delta_u^+$ at $z$. As the orientation of $\gamma^u$ is aligned with the orientation of $R_j$, $S_i$ leaves $\Delta_j^-$, which contradicts Claim 1.
\item $z=z_{2,3}=x_{i,n}$ for some $i<n$, then $u_i'\in\Delta_n^-$, which contradicts Claim 5.
\item $z=z_{3,1}=x_{2,n-1}$ implies $v_{2}'\in\Delta_n^-$.
\item $z=z_{3,2}=x_{i,n-1}$ is the crossing of $P_{n-1}$ and $S_i$, then $v_i'\in\Delta_n^-$.
\item $z=z_{3,3}=v$ is a touching of $\gamma^u$ and $\gamma^v$.
\end{enumerate}

Thus, $\gamma^u$ and $\gamma^v$ must cross at a new crossing and the statement of the theorem follows.
\end{proof}

\begin{figure}
\psfragscanon
\psfrag{A}{(a)}
\psfrag{B}{(b)}
\begin{center}
\includegraphics[width=135mm]{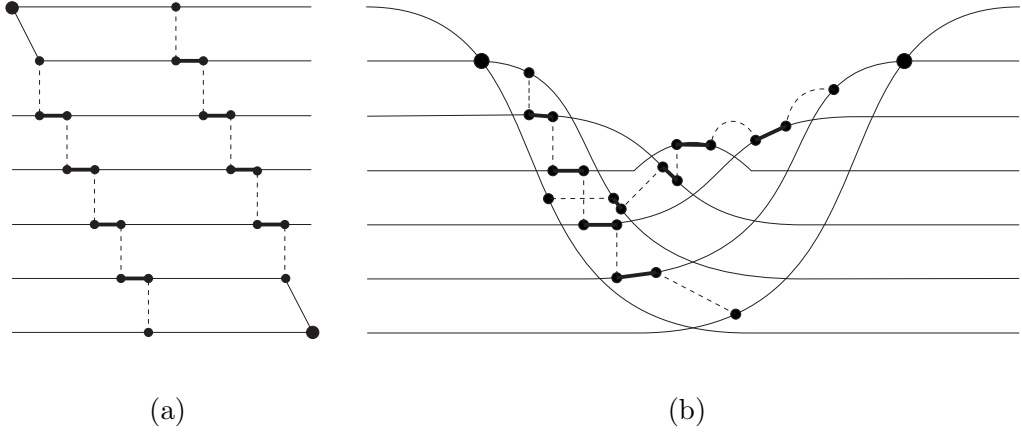}
\end{center}
\caption{(a) The tile $S_7$. (b) A tile drawing of $S_7$ with 20 crossings.}
\label{fg:staircase}
\end{figure}

The reader shall have no difficulty rigorously describing the tile $S_n$, $n\ge 3$, an example of which is for $n=7$ presented in Figure \ref{fg:staircase} (a). A \DEF{staircase tile} of width $n\ge 3$ is a tile obtained from $S_n$ by contracting some (possibly zero) thick edges of $S_n$. Such a tile is a perfect planar tile. A \DEF{staircase sequence} of width $n$ is a sequence of tiles of odd length in which staircase tiles of width $n$ alternate with inverted staircase tiles of width $n$. Any staircase sequence is a cyclically-compatible sequence of tiles. 

\begin{proposition}
\label{pr:graphsS}
Let $\cT$ be a staircase sequence of width $n$ and odd length $m\ge 4{n\choose 2}-5$. The graph $G=\circ (\cT^{\updownarrow})$ is a crossing-critical graph with $\crn(G)={n\choose 2}-1$.
\end{proposition}
\begin{proof}
A generalization of the drawing in Figure \ref{fg:staircase} demonstrates that $\tcrn(S_n)\le{n\choose 2}-1$. As $m$ is odd, the cut $\cT^{\updownarrow}/i$ contains a twisted staircase strip of width $n$ for any $i=0,\ldots,m-1$, and Theorem \ref{th:tcrnStrip} implies $\tcrn(\cT^{\updownarrow}/i)\ge {n\choose 2}-1$. Planarity of tiles $S_n$ and Lemma \ref{lm:join} establish equality and Corollary \ref{cr:generalTiles} implies $\crn(G)={n\choose 2}-1$. 

After removing any edge from $S_n$, we can decrease the number of crossings in the drawing in Figure \ref{fg:staircase} (b). Thus, $S_n$ is a $\left({n\choose 2}-1\right)$-degenerate tile and $\cT$ is a $\left({n\choose 2}-1\right)$-critical sequence; the criticality of $G$ follows by Corollary \ref{cr:criticalTiles}.
\end{proof}

Let $S_n'$ be the inverted tile $S_n$. Let $\cS_{n,m}$ be the staircase sequence $(S_n,S_n',S_n,S_n',\ldots,S_n)$ of odd length $m\ge 1$ and $\cS(n,m,c)$ the set of graphs obtained from $\circ(\cS_{n,m}^{\updownarrow})$ by contracting $c$ thick edges in the tiles of $\cS_{n,m}$. These graphs almost settle Question \ref{qu:salazar}:

\begin{proposition}[\cite{SACN,MAT}]
\label{pr:apbOdd}
Let $r=3+\frac{a}{b}$ with $1\le a < b$. If $a+b$ is odd, then, for $n\ge \max\left(\frac{5b-a}{2(b-a)},\frac{7a+b}{4a},4 \right)$, $m(t)=(2t + 1)(a + b)$, and $c(t)=(2t + 1)((4n - 7)a - b)$, the family $\cQ(a,b,n)=\bigcup_{t=n^2}^\infty \cS(n,m(t),c(t))$ contains $\left({n\choose 2}-1\right)$-crossing-critical graphs with average degree $r$.
\end{proposition}

Demanding the average degree of the graphs in $\cS(m,n,c)$ to be $r=3+\frac{a}{b}$, $1\le a<b$, $a+b$ even, forces $m(t)$ to be an even number and the resulting graphs are no longer critical.

\section{Zip product and criticality of graphs}
\label{sc:zip}

Zip product is an operation on graphs or their drawings that was used in \cite{CNCPP,CNCPT} to establish the crossing number of Cartesian products of several graphs with trees. For two graphs $G_i$ ($i=1,2$), their vertices $v_i$ of degree $d$ not incident with multiple edges (we call such vertices \DEF{simple}), and a bijection $\dfnc{\sigma}{N_1}{N_2}$ of the neighborhoods $N_i$ of $v_i$ in $G_i$, the \DEF{zip product} of the graphs $G_1$ and $G_2$ according to $\sigma$ is the graph $G_1\odot_\sigma G_2$, obtained from the disjoint union of $G_1-v_1$ and  $G_2-v_2$ after adding the edge $u\sigma(u)$ for every $u\in N_1$. We call $\sigma$ a \DEF{zip function} of the graphs.
Let $\GvvG$ denote the set of all pairwise nonisomorphic graphs, obtained as a zip product $\GsG$ for some bijection $\dfnc{\sigma}{N_1}{N_2}$. 

A drawing $D_i$ of a graph $G_i$ imposes a cyclic ordering of the edges incident with $v_i$, which defines a labeling $\dfnc{\pi_i}{N_i}{\lset{1}{d}}$ up to a cyclic permutation. A \DEF{zip function} of the drawings $D_1$ and $D_2$ at vertices $v_1$ and $v_2$ is $\dfnc{\sigma}{N_1}{N_2}$, $\sigma=\pi_2^{-1}\pi_1$. 

\begin{lemma}[\cite{CNCPP}]
\label{lm:upper}
For $i=1,2$, let $D_i$ be an optimal drawing of $G_i$, let $v_i\in V(G_i)$ be a simple vertex of degree $d$, and let $\sigma$ be a zip function of $D_1$ and $D_2$ at $v_1$ and $v_2$. Then, $\crn(G_1\odot_\sigma G_2)\le \crn(G_1)+\crn(G_2)$.
\end{lemma}

Let $v\in V(G)$ be a vertex of degree $d$ in $G$. A \DEF{bundle} of $v$ is a set $B$ of $d$ edge disjoint paths from $v$ to some vertex $u\in V(G)$, $u\neq v$. Two bundles $B_1$ and $B_2$ of $v$ are \DEF{coherent} if the sets of edges $E(B_1)\cap E(G-v)$ and $E(B_2)\cap E(G-v)$ are disjoint. 

\begin{lemma}[\cite{CNCPT}]
\label{lm:lower}
For $i=1,2$, let $G_i$ be a graph, $v_i\in V(G_i)$ its simple vertex of degree $d$, and $N_i=N_{G_i}(v_i)$. Also assume that $v_i$ has two coherent bundles in $G_i$. Then, $\crn(G_1\odot_\sigma G_2)\ge\crn(G_1)+\crn(G_2)$ for any bijection $\dfnc{\sigma}{N_1}{N_2}$. 
\end{lemma}

The following observations are useful in iterative applications of the zip product.
\begin{lemma}
\label{lm:semi}
Let $G_1$ and $G_2$ be disjoint graphs, $v_i\in V(G_i)$ simple, $\deg_{G_i}(v_i)=d$, and $G\in \GvvG$.
\begin{enumerate}[label=(\roman{*}), ref=(\roman{*})]
\item \label{it:semi}If $v_2$ has a bundle in $G_2$ and $v\in V(G_1)$ has $k$ pairwise coherent bundles in $G_1$, then $v$ has $k$ pairwise coherent bundles in $G$.
\item \label{it:conn} If, for $i=1,2$, the graph $G_i$ is $k_i$-connected, $k_i\ge 2$, then $G$ is $k$-connected for $k=\min(k_1,k_2)$.
\item \label{it:econn} If, for $i=1,2$, the graph $G_i$ is $k_i$-edge-connected, $k_i\ge 2$, then $G$ is $k$-edge-connected for $k=\min(k_1,k_2)$.
\end{enumerate}
\end{lemma}
\begin{proof}
\textbf{\ref{it:semi}}: See \cite{CNCPT}.

\textbf{\ref{it:conn}}: Let $S\subseteq V(G)$ be a separator of $G$. If $S\subseteq V(G_i-v_i)$, then, as $G_{3-i}-v_{3-i}$ is nonempty and $(k_{3-i}-1)$-connected, $S$ is a separator in $G_i$ and $|S|\ge k$. Let $S_i=S\cap V(G_i-v_i)$ and $S_i\neq \emptyset$ for $i=1,2$. If $S_i\cup \{v_i\}$ is a separator in $G_i$ for one of $i=1,2$, then $|S|\ge k$. Otherwise, the vertices of $G_i-v_i - S$ are all in the same component of $G-S$ for both $i=1,2$, thus $|S|\ge d\ge k$.

\textbf{\ref{it:econn}}: The argument is similar to (ii).
\end{proof}

Let $S\subset V(G)$ be a set and $\Gamma\subseteq \operatorname{Aut}(G)$ a group. We say that $S$ is \DEF{$\Gamma$-homogeneous} in $G$ if any permutation $\pi$ of $S$ can be extended to an automorphism $\sigma\in\Gamma$. For $S\subseteq V(G)$, let $\Gamma(S)$ be the pointwise stabilizer of $S$ in $\operatorname{Aut}(G)$.  We say that a vertex $v\in V(G)$ has a \DEF{homogeneous neighborhood} in $G$ if $N_G(v)$ is $\Gamma(\{v\})$-homogeneous in $G$. 

If all the vertices in $N_G(v)$ have the same set of neighbors for a vertex $v\in V(G)$, then $v$ has a homogeneous neighborhood $G$. Thus, every vertex of a complete or complete bipartite graph $K$ has such neighborhood in $K$.

\begin{lemma}
\label{lm:homoupper}
For $i=1,2$, let $G_i$ be a graph with a simple vertex $v_i\in V(G_i)$ of degree $d$. If $d=3$ or $v_2$ has a homogeneous neighborhood in $G_2$, then $\crn(G)\le\crn(G_1)+\crn(G_2)$ for every $G\in\GvvG$.
\end{lemma}
\begin{proof}
Assume $N_1=N_{G_1}(v_1)$, $N_2=N_{G_2}(v_2)$, and let the zip function of $G$ be $\dfnc{\sigma}{N_1}{N_2}$. For $i=1,2$, let $D_i$ be an optimal drawing of $G_i$ and let $\dfnc{\pi_i}{N_i}{N_i}$ denote the vertex rotation around $v_i$ in $D_i$. For $d>3$, there exists an automorphism $\rho\in\Gamma_{G_2}(\{v_2\})$ with $\rho/N_2=\sigma\pi_1\sigma^{-1}\pi_2^{-1}$. Applying $\rho$ to $D_2$ produces a drawing $D_2'$ with vertex rotation $\rho\pi_2=\sigma\pi_1\sigma^{-1}$ around $v_2$. Since $\sigma^{-1}(\rho\pi_2)^{-1}\sigma\pi_1=\id$, $\sigma$ is a zip function of $D_1$ and $D_2'$. If $d=3$, then $\sigma$ is a zip fucntion of $D_1$ and either $D_2$ or its mirrored image $D_2'$. The claim follows by Lemma \ref{lm:upper}.
\end{proof}

\begin{lemma}
\label{lm:homoextends}
For $i=1,2$, let $G_i$ be a graph with a simple vertex $v_i\in V(G_i)$ of degree $d$. Assume that $v_2$ has a homogeneous neighborhood in $G_2$. If, for some vertex $v\in G_1$, $v\ne v_1$, its neighborhood $N=N_{G_1}(v)$ is $\Gamma_{G_1}(\{v,v_1\})$-homogeneous, then $v$ has a homogeneous neighborhood in $G\in\GvvG$.
\end{lemma}
\begin{proof}
Assume $N_1=N_{G_1}(v_1)$, $N_2=N_{G_2}(v_2)$, and let the zip function of $G$ be $\dfnc{\sigma}{N_1}{N_2}$. For a permutation $\pi$ of $N$, there exists $\sigma_1\in\Gamma_{G_1}(\{v,v_1\})$, such that $\sigma_1/N=\pi$. Let $\pi_1=\sigma_1/N_1$, and set $\pi_2=\sigma\pi_1\sigma^{-1}$. As $v_2$ has a homogeneous neighborhood, there exists an automorphism $\sigma_2\in\Gamma_{G_2}(v_2)$ with $\sigma_2/N_2=\pi_2$. It is easy to verify that a function $\dfnc{\Phi}{G}{G}$ with $\Phi/(G_i-v_i)=\sigma_i/(G_i-v_i)$ for $i=1,2$, is an automorphism of $\Gamma_{G}(v)$, for which $\Phi/N=\pi$. Thus, $v$ has a homogeneous neighborhood in $G$. 
\end{proof}

\begin{theorem}
\label{th:crithomo}
For $i=1,2$, let $G_i$ be a $k_i$-crossing-critical graph with a simple vertex $v_i\in V(G_i)$ of degree $d$. If $d\neq 3$, then let $v_i$ have a homogeneous neighborhood. If $\crn(G)\ge k$ for  $k=\max\dset{\crn(G_i)+k_{3-i}}{i\in\{1,2\}}$ and $G\in\GvvG$, then $G$ is $k$-crossing-critical.
\end{theorem}
\begin{proof}
Again, assume $N_1=N_{G_1}(v_1)$, $N_2=N_{G_2}(v_2)$, and let the zip function of $G$ be $\dfnc{\sigma}{N_1}{N_2}$. Let $e\in E(G)$, and assume $e\in E(G_1-v_1)$. Let $D_1$ be an optimal drawing of $G_1-e$ and $D_2$ an optimal drawing of $G_2$. We adjust $D_2$ either using an appropriate automorphism in $\Gamma_{G_2}(\{v_2\})$ for $d>3$ or mirroring for $d=3$ similarly as in the proof of Lemma \ref{lm:homoupper} and combine $D_2$ with $D_1$ to produce  a drawing of $G-e$ with at most $k$ crossings. Similar arguments apply for $e\in E(G_2-v_2)$. 

If $e=v\sigma(v)$ for $v\in N_1$, let $D_1$ be an optimal drawing of $G_1-vv_1$ and $D_2$ an optimal drawing of $G_2-v_2\sigma(v)$. If $d=3$, we can clearly combine $D_1$ and $D_2$ into a drawing of $G$ with at most $k$ crossings. Otherwise, let $\dfnc{\pi_i}{N_i}{N_i}$ be the vertex rotation around $v_i$ in $D_i$ and $\rho\in\Gamma_{G_2}(\{v_2\})$ an automorphism of $G_2$ with $\rho/(N_2\setminus\{\sigma(v)\})=\sigma\pi_1\sigma^{-1}\pi_2^{-1}$. The vertices of $N_2$ can be rearranged with $\rho$ as in the proof of Lemma \ref{lm:homoupper}, thus $G-e$ can be drawn with at most $k_1+k_2$ crossings. 
\end{proof}

Lemma \ref{lm:lower} states that two coherent bundles at each $v_i$ are a sufficient condition for $\crn(G)\ge k$ in Theorem \ref{th:crithomo}. 

Argument of Theorem \ref{th:crithomo} has a generalization to (not necessarily critical) graphs that have a special vertex cover. Let $G$ be a graph and $S=\lset{v_1}{v_t}\subseteq V(G)$. For each $v_i\in S$, let $G_i$ be a graph and let $u_i\in V(G_i)$ be a simple vertex of degree $d(u_i)=d(v_i)$ having two coherent bundles in $G_i$. Let $\cS:=\dset{(v_i,G_i)}{i\in\lset{1}{t}}$. The family $G^\cS:=\Gamma^t$ is defined inductively as follows: $\Gamma^0=\{G\}$, and, for $i=1,\ldots,t$, let $\Gamma^i:=\bigcup_{H\in\Gamma^{i-1}}\vzip{H}{v_i}{u_i}{G_i}$. Further, let $\cS_i:=\cS\setminus \{(v_i,G_i)\}$.

\begin{theorem}
\label{th:crittos}
Let $G$ be a graph, $S$ be its vertex cover consisting of simple vertices of degree three each having two coherent bundles, and $\cS$ be defined as above. If each graph $G_i$ is $k_i$-crossing-critical for $i=1,\ldots,t$, then every $\bar G\in G^\cS$ is $k$-crossing-critical for $k=\max\dset{\crn(\bar G)-\crn(G_i)+k_i}{i\in\lset{1}{t}}$ and has crossing number $\crn(\bar G)=\crn(G)+\sum_{i=1}^t\crn(G_i)$.
\end{theorem}
\begin{proof}
Iterative application of Lemmas \ref{lm:lower}, \ref{lm:semi} \ref{it:semi}, and \ref{lm:homoupper} implies $\crn(\bar G)=\crn(G)+\sum_{i=1}^t\crn(G_i)$. To establish criticality of $\bar G$, let $e\in E(G^\cS)$ be an arbitrary edge and let $\bar G_j\in G^{\cS_j}$, $j=1,\ldots,t$, be the graph, such that $\bar G\in\vzip{\bar G_j}{v_j}{u_j}{G_j}$. 

Case 1: Assume $e\in E(G_j-v_j)$ for some $j\in\lset{1}{t}$. Let $D_1$ be an optimal drawing of $\bar G_j$ with $v_j$ in the infinite face and let $D_2$ be an optimal drawing of $G_j-e$ with $u_j$ in the infinite face. We can combine $D_1-v_j$ and $D_2-u_j$ into a drawing $D$ of $G-e$. By Lemma \ref{lm:homoupper}, $D$ has at most $\crn(G)+k_j+\sum_{i\neq j}\crn(G_i)\le k$ crossings.

Case 2: Assume $e\not\in E(G_i-v_i)$ for any $i\in\lset{1}{t}$. As $S$ is a vertex cover in $G$, there exists $j\in\lset{1}{t}$, such that $e$ connects some neighbor $x$ of $u_j\in V(G_j)$ with some neighbor $y$ of $v_j$ in $\bar G_j$. Let $e_1=v_jy\in E(\bar G_j)$, $e_2=u_jx\in E(G_j)$, and let $D_1$ be an optimal drawing of $\bar G_j-e_1$ with $v_j$ on the infinite face and $D_2$ an optimal drawing of $G_j-e_2$ with $u_j$ in the infinite face. We can combine $D_1-v_j$ and $D_2-u_j$ into a drawing $D$ of $G-e$. By Lemma \ref{lm:homoupper}, $D$ has at most $\crn(G-e)+k_j+\sum_{i\neq j}\crn(G_i)\le k$ crossings.
\end{proof}

\begin{figure}
\begin{center}
\includegraphics[height=44mm]{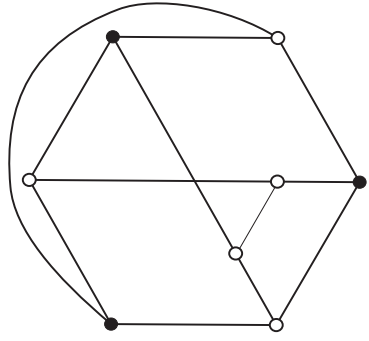}
\caption{Graph with a vertex cover of cubic vertices each having two coherent bundles.}
\label{fg:manybundles}
\end{center}
\end{figure}

Lea\~{n}os and Salazar established a decomposition of $2$-connected crossing-critical graphs into smaller $3$-connected crossing-critical graphs in \cite{LS}. Theorem \ref{th:crittos}, in combination with the graph in Figure \ref{fg:manybundles}, which has a vertex cover consisting of cubic vertices with two coherent bundles but is not crossing-critical, suggests that a similar decomposition does not exist for $3$-connected crossing-critical graphs.

For $d,d'\ge 3$, let $K_{d,d'}$ be a properly 2-colored complete bipartite graph: vertices of degree $d$ are colored black and vertices of degree $d'$ are colored white. For $p\ge 1$, let the family $\cR(d,d',p)$ consist of graphs with 2-colored vertices, obtained as follows: $\cR(d,d',1)=\{K_{d,d'}\}$ and $\cR(d,d',p)=\bigcup_{G\in \cR(d,d',p-1)}\vzip{G}{v_1}{v_2}{K_{d,d'}}$, where $v_1$ (respectively, $v_2$) is a black vertex in $G$ ($K_{d,d'}$). If $d=d'=3$, we allow $v_i$ to be any vertex. We preserve the colors of vertices in the zip product, thus the graphs in $\cR(d,d',p)$ are not properly colored for $p\ge 2$.

\begin{proposition}
\label{pr:graphsR}
Let $d, d'\ge 3$. Then every graph $G\in \cR(d,d',p)$ is a simple $3$-connected crossing-critical graph with $\crn(G)=p\crn(K_{d,d'})$.
\end{proposition}
\begin{proof}
By induction on $p$ and using Lemma \ref{lm:homoextends}, we show that all black vertices of $G$  have homogeneous neighborhoods. Iterative  application of Lemmas \ref{lm:lower}, \ref{lm:semi} \ref{it:semi}, \ref{lm:homoupper}, and Theorem \ref{th:crithomo} establish the crossing number of $G$ and its criticality. 
\end{proof}

Jaeger proved the following result:
\begin{theorem}[\cite{J}]
Every $3$-connected cubic graph with crossing number one has chromatic index three.
\end{theorem}

Graphs of the family $\cR(3,3,p)$ in a zip product with Petersen graph show that a similar result cannot be obtained for any crossing number greater than one. 

\begin{proposition}[\cite{SACN}]
For $k\ge 2$, there exist simple cubic $3$-connected crossing-critical graphs with crossing number $k$ and with no $3$-edge-coloring.
\end{proposition}

\section{The main construction}
\label{sc:families}

\begin{theorem}
\label{th:main}
Let $r\in (3,6)$ be a rational number and $k$ an integer. There exists a convex continuous function $\dfnc{f}{(3,6)}{\RR^+}$ such that, for $k\ge f(r)$, there exists an infinite family of simple $3$-connected crossing-critical graphs with average degree $r$ and crossing number $k$.
\end{theorem}
\begin{proof}
We present a constructive proof for $$f(r)=240 + \tfrac{512}{(6 - r)^2} + \tfrac{224}{6 - r} + \tfrac{25}{16(r - 3)^2} + \tfrac{40}{r - 3}.$$

A sketch of the construction is as follows: The graphs are obtained as a zip product of crossing-critical graphs from the families $\cS$ and $\cR$, and of the graphs $H$, all defined above. The graphs $H$ allow average degree close to six and the graphs from $\cS$ allow average degree close to three. A disjoint union of two such graphs consisting of a proportional number of tiles would have a fixed average degree and crossing number. The zip product compromises the pattern needed for fixed average degree, for which we compensate with the graphs from $\cR$. Their role is also to fine-tune the desired crossing number of the resulting graph. 

More precisely, let $\Gamma(n,m,c,w,s,p,q)$ be the family of graphs, constructed in the following way: first we combine $G_1\in\cS(n,m,c)$ and $G_2=H(w,s)$ in the family   $\Gamma(n,m,c,w,s,0,0)$ $=\bigcup_{G_1,G_2}\bigcup_{v_1,v_2}\GvvG$. Further, we combine $G_1\in \Gamma(n,m,c,w,s,0,0)$ and $G_2\in \cR(3,3,p)$ in the family $\Gamma(n,m,c,w,s,p,0)=\bigcup_{G_1,G_2}\bigcup_{v_1,v_2}\GvvG$. Finally, we combine the graphs $G_1\in \Gamma(n,m,c,w,s,p,0)$ and $G_2\in \cR(3,5,q)$ in the family $\Gamma(n,m,c,w,s,p,q)=\bigcup_{G_1,G_2}\bigcup_{v_1,v_2}\GvvG$. In each case, $v_i\in V(G_i)$ is any vertex of degree three, as all such vertices have two coherent bundles. Propositions \ref{pr:graphsH}, \ref{pr:graphsS}, and \ref{pr:graphsR} imply that the graphs used in construction are crossing-critical graphs whenever the following conditions are satisfied:
\begin{eqnarray}
\label{eq:cond1} n  & \ge & 3, \\
m &=&2m'+1, \\
m' &\ge& 2{n\choose 2},\\
c & \ge & 0,\\
c & \le & 2m(n-3),\\
w & \ge & 0, \\
s & \ge & 4(32 w^2+56w+ 31), \\
p & \ge & 1,\hbox{ and} \\
\label{eq:condN} q & \ge & 1.
\end{eqnarray}

Results in \cite{K29} establish $\crn(K_{3,5})=4$, thus Theorem \ref{th:crithomo} together with Lemmas \ref{lm:lower}, \ref{lm:semi} \ref{it:semi}, and \ref{lm:homoupper} implies that subject to (\ref{eq:cond1})--(\ref{eq:condN}) the graphs in $\Gamma(n,m,c,w,s,p,q)$ are crossing-critical with crossing number 
\begin{equation}
\label{eq:crnGamma}
k={n\choose 2}+32 w^2+56w+p+4q+ 30.
\end{equation}
Their average degree is 
\begin{equation}
\label{eq:avgDeg}
\bar d=6-\frac{4 (m' (6n-11) + 3 n + 3 p + 3 q + 4 s-c-7)}{2 m' (4 n-7 )+ 4 n+ 4 s w + 9 s  + 4 p + 6 q - c-9}.
\end{equation}

Using (\ref{eq:crnGamma}) we express $p$ in terms of $k$ and other parameters. We set $s$ and $m$ to be a linear function of a new parameter $t$, which will determine the size of the resulting graph. We substitute these values into (\ref{eq:avgDeg}). Using $c$ we eliminate all the terms in the denominator that are independent of $t$. Parameter $q$ plays the same role in the numerator. Then $t$ cancels and we set the coefficients of the linear functions to yield the desired average degree. Finally, parameters $n$, $w$, and the constant terms of the linear functions are selected to satisfy the constraints (\ref{eq:cond1})--(\ref{eq:condN}). A more detailed analysis might produce a smaller lower bound $f$, but one constant term was selected to be zero to simplify the computations. 

More precisely, let $r=3+\tfrac{a}{b}$, $0<a<3b$, and $k\ge f(r)$. Perform the following integer divisions:
\begin{eqnarray*}
b  & = & b'a+b_r,\\
b' & = & 4b''+b_r',\\
4b & = & \bar b (3b-a) + \bar b_r, \hbox{ and }\\
k-\tfrac{b''(b''+5)}{2}-8\bar b(4\bar b+7) & = & k'(2b''+5)+k_r.
\end{eqnarray*}
For some integer $t$ set 
\begin{eqnarray*}
n  & = & b''+4,\\
m_t & = & 2t(27b-9a-4\bar b_r)-2k'+3,\\
c & = & 2k'-12b''-6k_r-33,\\
w &=& \bar b,\\
s_t & = & 2t((4b''+9)a-b), \\
p &=& k-\left(\tfrac{b''(b''+23)}{2}+8\bar b(4\bar b+7)+4k_r+56\right), \hbox{ and }\\
q &=& 2b''+k_r+5.
\end{eqnarray*}

The family $\Gamma(a,b,k)=\bigcup_{t=k}^\infty\Gamma(n,m_t,c,w,s_t,p,q)$ is an infinite family of crossing-critical graphs with average degree $r$ and crossing number $k$. Verification of the constraints (\ref{eq:cond1})--(\ref{eq:condN}) for any $r\in(3,6)$ and $k\ge f(r)$ requires some tedious computation that is omitted here; an interested reader can find it in \cite{MAT}. The function $f$ is a sum of functions that are convex on $(3,6)$ and thus itself convex. The graphs of $\Gamma(a,b,k)$ are $3$-connected by Lemma \ref{lm:semi} \ref{it:conn}.
\end{proof}

The convexity of the function $f$ in Theorem \ref{th:main} implies $N_I=\max\{f(r_1),$ $f(r_2)\}$  is a universal lower bound on $k$ for rational numbers within any closed interval $I=[r_1,r_2]\subseteq (3,6)$.

\section{Structure of crossing-critical graphs}
\label{sc:aspects}

It has recently been established that all large $2$-crossing-critical graphs are obtained as cyclizations of long sequences, composed out of copies of a small number of different tiles \cite{BORSa,BORSb}. The construction of crossing-critical graphs using zip product demonstrates that no such classification of tiles can exist for $k\ge 4$: by a generalized zip product of a graph and a tile, as proposed in  \cite{SACN}, one can obtain an infinite sequence of $k$-degenerate tiles, all having the same tile crossing number. These tiles in combination with corresponding perfect planar tiles yield $k$-crossing-critical graphs. 

For $k$ large enough, one can obtain $k$-crossing-critical graphs from an arbitrary (not necessarily critical) graph that has a vertex cover consisting of simple vertices of degree three with two coherent bundles, cf.~Theorem \ref{th:crittos}. Figure \ref{fg:structure} sketches the described structure.

\begin{figure}
\begin{center}
\includegraphics[width=65mm]{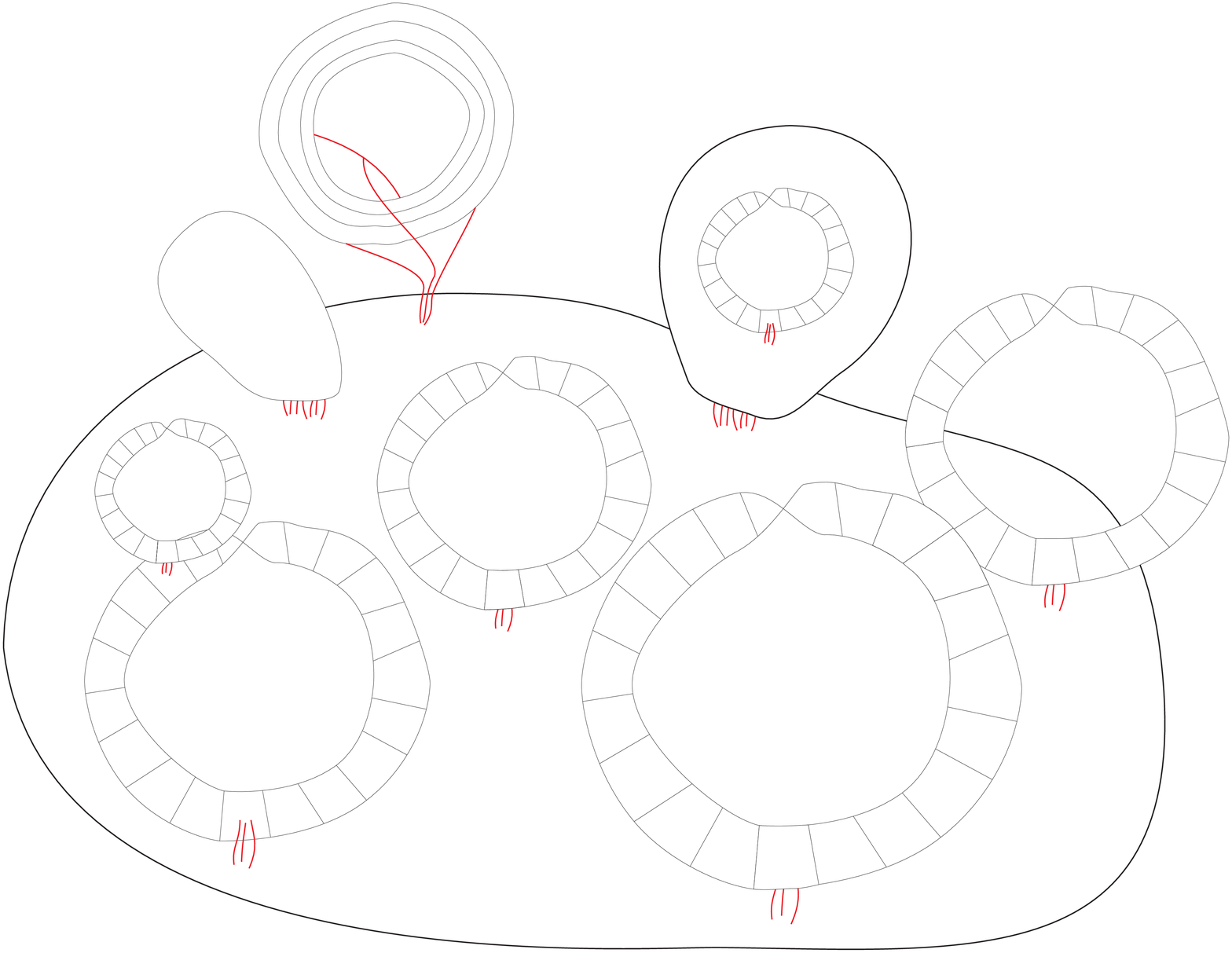}
\end{center}
\caption{Structure of known large $k$-crossing-critical graphs.}
\label{fg:structure}
\end{figure}

The following questions remain open regarding the degrees of vertices in $k$-crossing-critical graphs: 

\begin{question}[\cite{RT164}]
\label{qu:5reg}
Do there exist an integer $k>0$ and an infinite family of (simple) $5$-regular $3$-connected $k$-crossing-critical graphs?
\end{question}

\begin{question}
\label{qu:deg6}
Do there exist an integer $k>0$ and an infinite family of (simple) $3$-connected $k$-crossing-critical graphs of average degree six?
\end{question}

Arguments of \cite{RT164} used to establish that, for $k>0$, there exist only finitely many $k$-crossing-critical graphs with minimum degree six extend to graphs with a bounded number of vertices of degree smaller than six. Thus, we may assume that a family positively answering Question \ref{qu:deg6} would contain graphs with arbitrarily many vertices of degree larger than six. But only vertices of degrees three, four, or six appear arbitrarily often in the graphs of the known infinite families of $k$-crossing-critical graphs. We thus propose the following question, an answer to which would be  a step in answering Questions \ref{qu:5reg} and \ref{qu:deg6}.

\begin{question}
Does there exist an integer $k>0$, such that, for every integer $n$, there exists a $3$-connected $k$-crossing-critical graph $G_n$ with more than $n$ vertices of degree distinct from three, four and six?
\end{question}

We can obtain arbitrarily large crossing-critical graphs with arbitrarily many vertices of degree $d$, for any $d$, by applying the zip product to graphs $K_{3,d}$, $K_{d,d}$, and the graphs from the known infinite families. However, the crossing numbers of these graphs grow with the number of such vertices.

\section{Acknowledgement}

The author would like to express gratitude to Bojan Mohar for pointing out the interesting subject and numerous discussions about it. Also, Matt DeVos merits credit for several valuable terminological suggestions.


\begin{thebibliography}{99}
\bibitem{CNCPP} D.~Bokal, On the crossing number of Cartesian products with paths, J.~Combin.\ Theory Ser.~B. 97 (2007), 381--384.
\bibitem{CNCPT} D.~Bokal, On the crossing number of Cartesian products with trees, J.~Graph Theory 56 (2007) 287--300.
\bibitem{SACN}{D.~Bokal,
Structural approach to the crossing number of graphs, Dissertation, University of Ljubljana, 2006.}
\bibitem{MAT} D.~Bokal, Technical details regarding infinite families of crossing-critical graphs with prescribed average degree and crossing number, Mathematica$^{\hbox{\tiny TM}}$ notebook (2006),\\  \texttt{http://arxiv.org/abs/0909.1939}.
\bibitem{BORSa} D.~Bokal, B.~Oporowski, R.~B.~Richter, G.~Salazar, Characterization of 2-crossing-critical graphs I: Low connectivity or no $V_{2n}$-minor, submitted.
\bibitem{BORSb} D.~Bokal, B.~Oporowski, R.~B.~Richter, G.~Salazar, Characterization of 2-crossing-critical graphs II: the tiles, in preparation.
\bibitem{D} R.~Diestel, Graph Theory, Graduate Texts in Mathematics, vol.~173, Springer Verlag, New York, 2000.
\bibitem{DMLMD} Z.~Dvo\v{r}\'{a}k, B.~Mohar, Crossing-critical graphs with large maximum degree, submitted, \texttt{http://arxiv.org/PS\_cache/arxiv/pdf/0907/0907.1599v1.pdf}.
\bibitem{GRS438}{J.~F.~Geelen, R.~B.~Richter, G.~Salazar,
Embedding grids on surfaces, European J. Combin. 25 (2004), 785--792.}
\bibitem{H351} P.~Hlin\v{e}n\'{y}, Crossing-critical graphs and path-width, in: {\em Proc. $9^{\hbox{th}}$ Intl.~Symposium on Graph Drawing}, Lecture Notes in Computer Science 2265, Springer Verlag, Berlin, 2001, 102--114.
\bibitem{H410} P.~Hlin\v{e}n\'{y}, Crossing-critical graphs have bounded path-width, J.~Combin.\ Theory Ser.~B 88 (2003), 347--367.
\bibitem{HEJC} P.~Hlin\v{e}n\'{y}, New infinite families of almost-planar crossing-critical graphs, Electron.\ J.\ Combin. 15 (2008), \#R102.
\bibitem{J}{F.~Jaeger, Tait's theorem for graphs with crossing number at most one, Ars Combin.~9 (1980), 283--287.}
\bibitem{K29} D.~J.~Kleitman, The crossing number of $K_{5,n}$, J.~Combin.\ Theory Ser.~B 9 (1971), 315--323.
\bibitem{K119} M.~Kochol, Construction of crossing-critical graphs, Discrete Math.\ 66 (1987), 311--313.
\bibitem{LS}{J.~Lea\~{n}os, G.~Salazar, On the additivity of crossing numbers of graphs, J. Knot Theory Ramifications 17 (2008), 1043--1050.}
\bibitem{PR406A} B.~Pinontoan, R.~B.~Richter, Crossing numbers of sequences of
graphs I: general tiles, Australas. J. Combin. 30 (2004), 197–-206.
\bibitem{PR406} B.~Pinontoan, R.~B.~Richter, Crossing numbers of sequences of graphs II: planar tiles, J.~Graph Theory 42 (2003), 332--342.
\bibitem{RT198}{ R.~B.~Richter, C.~Thomassen, Intersections of curve systems and the crossing number of $\crt{C_5}{C_5}$, Discrete Comput.\ Geom.~13 (1995), 149--159.}
\bibitem{RT164} R.~B.~Richter, C.~Thomassen, Minimal graphs with crossing number at least $k$, J.~Combin.\ Theory Ser.~B 58 (1993), 217--224.
\bibitem{S395} G.~Salazar, Infinite families of crossing-critical graphs with given average degree, Discrete Math. 271 (2003), 343--350.
\bibitem{S}{M.~Lomel\'i, G.~Salazar, Nearly-light cycles in embedded graphs and  crossing-critical graphs, J.~Graph Theory 53 (2006), 151--156.}
\bibitem{S297}{ G.~Salazar, On a crossing number result of Richter and Thomassen, J.~Combin.\ Theory Ser.~B 79 (2000), 98--99.}
\bibitem{S96}{ J.~\v{S}ir\'{a}\v{n}, Crossing-critical edges and Kuratowski subgraphs of a graph, J.~Combin.\ Theory Ser.~B 35 (1983), 83--92.}
\bibitem{S100} J.~\v{S}ir\'{a}\v{n}, Infinite families of crossing-critical graphs with a given crossing number, Discrete Math.\ 48 (1984), 129--132.
\end{thebibliography}
\end{document}